\documentclass[11pt]{article}
\usepackage[utf8]{inputenc}
\usepackage{graphicx} 
\usepackage{amsmath,amsfonts,amssymb,amsthm}
\usepackage[mathscr]{eucal}
\usepackage{amscd}
\usepackage{tikz}
\usepackage{tkz-graph}
\usepackage{multicol}
\usepackage{float}
\usepackage{authblk}
\usepackage{titlesec}

\titleformat{\section}
  {\normalfont\Large\bfseries\centering}{\thesection}{1em}{}
\titleformat{\subsection}{\normalfont\large\bfseries\centering}{\thesubsection}{1em}{}

\newtheorem{theorem}{Theorem}[section]

\newtheorem{proposition}[theorem]{Proposition}
\newtheorem{corollary}[theorem]{Corollary}
\newtheorem{conjecture}[theorem]{Conjecture}
\newtheorem{remark}[theorem]{Remark}

\title{The $p$-rationality of $\mathbb{Q}\left(\sqrt{-(kp+m)}\right)$ and $\mathbb{Q}\left(\sqrt{p(p+1)}\right)$}
\author{Chen Lin$^1$\footnote{Corresponding author.}, Xuejun Guo$^2$}
\affil{{\small {$^{1,2}$School of Mathematics, Nanjing University, Nanjing 210093, China}\\
$^1$chen.lin@smail.nju.edu.cn, 
$^2$guoxj@nju.edu.cn} }
\date{}
\date{}

\begin{document}
\maketitle
\begin{abstract}
In this paper, we construct new families of imaginary and real quadratic fields that are $p$-rational.

In the imaginary case, we prove that for any positive integer $k$ and any integer $m$, the imaginary quadratic field $\mathbb{Q}\left(\sqrt{-(kp+m)}\right)$ is $p$-rational for sufficiently large primes $p$. The proof relies on Louboutin's bound on the class numbers of imaginary quadratic fields. As a corollary, we recover the $p$-rationality of consecutive quadratic fields, a result due to Chattopadhyay, Laxmi and Saikia \cite{CLS}.

In the real case, we give an explicit proof of the $p$-rationality of the real quadratic field $\mathbb{Q}\left(\sqrt{p(p+1)}\right)$ for any odd prime $p$, and obtain new  pairs of real quadratic fields $\left(\mathbb{Q}\left(\sqrt{p(p-2)}\right),\allowbreak\mathbb{Q}\left(\sqrt{p(p-1)}\right)\right)$ and $\left(\mathbb{Q}\left(\sqrt{p(p+1)}\right),\mathbb{Q}\left(\sqrt{p(p+2)}\right)\right)$ for any prime $p>3$. We also construct new examples of $p$-rational triquadratic fields.
\end{abstract}
  \noindent { 2020\it Mathematics Subject Classification:} 11R29, 11R11\\[1mm]
    \noindent {\it Keywords: }$p$-rational fields, Greenberg's conjecture.

\section{Introduction}
Let $K$ be a number field and $p$ be a prime number. Let $S$ be the set of primes lying above $p$, and $K_S$ be the maximal $p$-extension of $K$ which is unramified outside the primes above $S$. We say the number field $K$ is \textit{$p$-rational} if the Galois group $\mathrm{Gal}(K_S/K)$ is a free pro-$p$-group.

Movahhedi \cite{Mova1988,Mova} introduced the concept of $p$-rationality to construct infinitely many non-abelian number fields that satisfy Leopoldt's conjecture at the prime $p$. Recently, Greenberg \cite{Greenberg} revisited this concept in his study of some representations of the absolute Galois group $G_{\mathbb{Q}}$ of $\mathbb{Q}$ with large image, and proposed the following conjecture, which was later extended by Cornut and Ray \cite{CR}. 

\begin{conjecture}
    For any odd prime $p$ and any positive integer $t$, there exists a $p$-rational number field $K$ whose Galois group over $\mathbb{Q}$ is isomorphic to $(\mathbb{Z}/2\mathbb{Z})^t$.
\end{conjecture}

This conjecture has been investigated for the cases $t\leq3$ over the past few years \cite{BR,Gras,BM2021,Ko,CLS,Byeon}. Benmerieme and Movahhedi \cite{BM2021} proved that for each odd prime $p$, there exist infinitely many $p$-rational real quadratic fields. They also constructed explicit imaginary and real $p$-rational bi-quadratic fields for each odd prime $p$, thereby resolving Greenberg's conjecture for $t=1$ and $t=2$. Koperecz \cite{Ko} later proved the case $t=3$ for infinitely many primes $p$. Byeon \cite{Byeon} proved the existence of infinitely many triquadratic 3-rational fields. Chattopadhyay, Laxmi and Saikia \cite{CLS} investigated $p$-rationality for certain consecutive imaginary and real quadratic fields. 

The aim of this paper is to present new families of imaginary and real $p$-rational quadratic fields. The structure of this paper is as follows.

In Section \ref{section-(kp+m)}, we prove that for any positive integer $k$ and any integer $m$, the imaginary quadratic field $\mathbb{Q}\left(\sqrt{-(kp+m)}\right)$ is $p$-rational for sufficiently large primes $p$ by using Louboutin's bound for class numbers of imaginary quadratic fields. As a corollary, we recover the simultaneous $p$-rationality of consecutive imaginary quadratic fields $\mathbb{Q}\left(\sqrt{-(p-1)}\right),\cdots,\mathbb{Q}\left(\sqrt{-(p-k)}\right)$, a result first shown by Chattopadhyay, Laxmi and Saikia \cite{CLS}.

In Section \ref{sectionp(p+1)}, we establish the $p$-rationality of the real quadratic number field $\mathbb{Q}\left(\sqrt{p(p+1)}\right)$ for any odd prime $p$. Using this result, we construct pairs of $p$-rational real quadratic fields in arithmetic progression, analogous to the examples provided in \cite{CLS}. Additionally, by applying Koperecz's method \cite{Ko}, we show that given a positive integer $\beta$ such that $\mathbb{Q}\left(\sqrt{-\beta}\right)$ is $p$-rational, the triquadratic number field $\mathbb{Q}\left(\sqrt{p(p+1)},\allowbreak \sqrt{p(p-1)},\allowbreak\sqrt{-\beta}\right)$ is $p$-rational for infinitely many primes $p$, which yields a new family of triquadratic $p$-rational fields.

\medskip
\textbf{Notations.} The following notations will be used throughout this paper:
\begin{itemize}
    \item $p$ always denotes an odd prime number;
    \item For a number field $K$, we denote its discriminant by $d_K$ and its class number by $h_K$;
    \item For a prime ideal $\mathfrak{p}$ of $K$, $K_{\mathfrak{p}}$ denotes the completion of $K$ at $\mathfrak{p}$, $N(\mathfrak{p})$ denotes the norm of $\mathfrak{p}$, and $U_{\mathfrak{p}}^{(i)}$ denotes the $i$\textsuperscript{th} higher unit group in the local field $K_{\mathfrak{p}}$;
    \item $f \ll g$ means $|f| \le cg$ for some unspecified positive constant $c$;
    \item $L(s,\chi_K)$ denotes the Dirichlet $L$-function associated with the character $\chi_K$ of $K$.
\end{itemize}

\section{The $p$-rationality of $\mathbb{Q}\left(\sqrt{-(kp+m)}\right)$}
\label{section-(kp+m)}
We begin by establishing the following result for imaginary quadratic fields.

\begin{theorem}
\label{kp+m}
    Given a positive integer $k$ and an integer $m$, the imaginary quadratic field $\mathbb{Q}\left(\sqrt{-(kp+m)}\right)$ is $p$-rational for sufficiently large primes $p$ (which ensures $kp+m>0$).
\end{theorem}

To prove this, we use a sufficient criterion for $p$-rationality, which follows from results of Greenberg \cite{Greenberg}, Benmerieme and Movahhedi \cite{BM2021}.

\begin{proposition}[\cite{Greenberg} Proposition 4.1 (i) or \cite{BM2021} Corollary 2.7]
    \label{criterionforimaginary}
    Let $K$ be an imaginary quadratic number field with class number $h_K$ and $p$ be a prime number (where $p\geq5$, or $p=3$ and is unramified in $K/\mathbb{Q}$). If $h_K$ is not divisible by $p$, then $K$ is $p$-rational.
\end{proposition}

This proposition reduces our task to show that $p$ does not divide the class number $h_K$. To this end, we require Louboutin's bound for class numbers of imaginary quadratic fields. 

\begin{proposition}[\cite{Louboutin}, Proposition 2]
\label{hk for imaginary field}
    Let $K$ be an imaginary quadratic field with discriminant $d_K$ and class number $h_K$. Then we have
    $$h_K\leq\frac{\omega_K\sqrt{|d_K|}}{4\pi}\left(\log |d_K|+\frac{3}{2}\right),$$
    where $\omega_K$ is the number of roots of unity in $K$.
\end{proposition}

With these results in hand, we are now ready to prove Theorem \ref{kp+m}.

\begin{proof}[Proof of Theorem \ref{kp+m}]
    Since $k$ is a positive integer and $m$ is a fixed integer, for sufficiently large primes $p$, we have $kp+m > 0$, ensuring that $K = \mathbb{Q}\left(\sqrt{-(kp+m)}\right)$ is strictly an imaginary quadratic field. Then the discriminant of $K$ satisfies
    $$|d_K|\leq4\times\text{square-free part of }(kp+m)\leq4(kp+m).$$
    Then by Proposition \ref{hk for imaginary field}, we have 
    $$h_K\leq\frac{\omega_K\sqrt{4(kp+m)}}{4\pi}\left(\log (4(kp+m))+\frac{3}{2}\right).$$
    Since the number of roots of unity in a quadratic field is 2,4 or 6, we have 
    $$h_K\leq\frac{3\sqrt{(kp+m)}}{\pi}\left(\log (kp+m)+\log 4+\frac{3}{2}\right)\ll \sqrt{p}\log p,$$
    with absolute and effective implied constant. 
    
    Hence, for sufficiently large prime $p$, the class number $h_K<p$, which implies $p\nmid h_K$. By the criterion in Proposition \ref{criterionforimaginary}, we conclude that $K$ is $p$-rational.
\end{proof}

As an application of Theorem \ref{kp+m}, we revisit the problem of consecutive $p$-rational quadratic fields, which was investigated by  Chattopadhyay, Laxmi and Saikia in \cite{CLS}. They established a result which can be derived as an immediate corollary from our theorem. 

\begin{corollary}[\cite{CLS}, Theorem 1.1]
\label{CLSThm1.1}
    Given a positive integer $t$, there exist infinitely many primes $p$ such that the consecutive imaginary quadratic fields $\mathbb{Q}\left(\sqrt{-(p-1)}\right),\cdots,\allowbreak\mathbb{Q}\left(\sqrt{-(p-t)}\right)$ are simultaneously $p$-rational.
\end{corollary}

\begin{proof}
    This follows immediately from Theorem \ref{kp+m}. Fix $k=1$, then for $m=-1,-2,\cdots,\allowbreak-t$, the quadratic number field $\mathbb{Q}\left(\sqrt{-(p+m)}\right)$ is $p$-rational for any prime $p>p_m$, where $p_m$ is a positive integer. 

    Let $P$ be the maximal element in the set $\{p_{-1},\cdots,p_{-t}\}$. Then the consecutive imaginary quadratic fields $\mathbb{Q}\left(\sqrt{-(p-1)}\right),\cdots,\mathbb{Q}\left(\sqrt{-(p-t)}\right)$ are simultaneously $p$-rational for any prime $p>P$.
\end{proof}
\begin{remark}
    Our proof of Corollary \ref{CLSThm1.1}, which is Theorem 1.1 of \cite{CLS}, is different from the original. The proof in \cite{CLS} relies on the existence of infinitely many primes $p$ such that $p-1,\cdots,p-t$ simultaneously have a divisor $\ell$ satisfying $\ell>(\log p)^2$. Our approach does not require this analytic assumption, thus simplifying the argument. 

    Furthermore, Theorem \ref{kp+m} constructs additional families of examples relevant to Question 1 in \cite{CLS}, which asks for the existence of $t$ consecutive $p$-rational fields. For instance, our theorem implies that the imaginary quadratic fields $\mathbb{Q}\left(\sqrt{-(kp+1)}\right),\cdots,\allowbreak\mathbb{Q}\left(\sqrt{-(kp+t)}\right)$ are also simultaneously $p$-rational for sufficiently large primes $p$.
\end{remark}

\section{The $p$-rationality of a real quadratic field}
\label{sectionp(p+1)}

We now turn to the case of real quadratic fields. The criterion for $p$-rationality in this setting is different from that for imaginary ones.

\begin{proposition}[\cite{Greenberg} Proposition 4.1(ii) or \cite{BM2021} Corollary 2.6]
    \label{criterionforreal}
    Let $K$ be a real quadratic number field and $p$ be a prime number (where $p\geq5$, or $p=3$ and is unramified in $K/\mathbb{Q}$). Then $K$ is $p$-rational if and only if $p$ does not divide the class number $h_K$ of $K$ and the fundamental unit of $K$ is not a $p$\textsuperscript{th} power in the completion $K_\mathfrak{p}$ for some prime $\mathfrak{p}$ lying above $p$.
\end{proposition}

We also introduce a useful proposition for testing the 3-rationality
of a specific class of real quadratic fields.

\begin{proposition}[\cite{BM2021} Proposition 2.9 (ii)]
    \label{criterionfor3}
    Let $F=\mathbb{Q}(\sqrt{d})$ be a quadratic field with $d \neq -3$ square-free. If $d \not\equiv -3 \pmod 9$, then $F$ is $3$-rational if and only if the $3$-primary part of the class group of its mirror field $F'=\mathbb{Q}\left(\sqrt{-3d}\right)$ is trivial.
\end{proposition}

By applying these criteria, we can establish the $p$-rationality of the field $\mathbb{Q}\left(\sqrt{p(p+1)}\right)$. We note that this specific real quadratic field was briefly mentioned in Benmerieme's thesis \cite{Benmerieme2021} as a subfield of the bi-quadratic extension $\mathbb{Q}\left(\sqrt{p(p+1)}, \sqrt{p(p-1)}\right)$. Here, we provide a direct and explicit proof of the following result. 

\begin{theorem}
    \label{p(p+1)}
    For any odd prime $p$, the real quadratic field $K=\mathbb{Q}\left(\sqrt{p(p+1)}\right)$ is $p$-rational.
\end{theorem}

\begin{proof}
    For $p=3$, the field reduces to $K=\mathbb{Q}\left(\sqrt{12}\right)=\mathbb{Q}\left(\sqrt{3}\right)$. Since the prime $3$ is ramified in $K$, we cannot apply Proposition \ref{criterionforreal}. Instead, we determine its $3$-rationality by using Proposition \ref{criterionfor3}. Since $d=3 \not\equiv -3 \pmod 9$, the $3$-rationality of $\mathbb{Q}\left(\sqrt{3}\right)$ is equivalent to the triviality of the $3$-primary part of the class group of its mirror field $K'=\mathbb{Q}\left(\sqrt{-3 \cdot 3}\right)=\mathbb{Q}(i)$. Since the class number of $\mathbb{Q}(i)$ is $1$, its $3$-primary part of the class group is automatically trivial. It follows that $K$ is $3$-rational.
    
    For the case $p > 3$, our proof strategy is inspired by the techniques used in \cite[Proposition 3.2]{BM2021}. We first prove the fundamental unit of $K$ is not a $p$\textsuperscript{th} power in the completion $K_\mathfrak{p}$ for some prime $\mathfrak{p}$ of $K$ lying above $p$.
    
    We claim that $\varepsilon = 2p+1+2\sqrt{p(p+1)}$ is the fundamental unit of the quadratic field $\mathbb{Q}\left(\sqrt{p(p+1)}\right)$ for $p>3$. Let $d$ denote the square-free part of $p(p+1)$. 
    
    CASE 1. If $d\equiv2,3\pmod{4}$, then the fundamental unit is of the form $x+y\sqrt{d}$, with $x,y$ positive integers. There would exist an integer $n \ge 2$ such that 
    \begin{align*}
        2(2p+1)=\mathrm{Tr}(\varepsilon)=\mathrm{Tr}((x+y\sqrt{d})^n)\geq2x^n+2x^{n-2}y^2d,
    \end{align*}
    where $\mathrm{Tr}:=\mathrm{Tr}_{K/\mathbb{Q}}$ denotes the trace map in ${\mathbb{Q}\left(\sqrt{p(p+1)}\right)/\mathbb{Q}}$. The inequality can be written as 
    \begin{align}
    \label{pinequality1}
        p\geq \frac{x^n-1}{2}+\frac{x^{n-2}y^2d}{2}.
    \end{align}
    Since $p\mid d$ and $x,y$ are positive integers, this inequality does not hold for $y\geq2$.
    
    If $x=y=1$, then we have $p\geq \frac{d}{2}$, that is $d=p$ or $2p$. However, this is impossible since neither $1+\sqrt{p}$ nor $1+\sqrt{2p}$ is even a unit.

    If $y=1,x\geq2$, then the inequality (\ref{pinequality1}) is $p\geq \frac{x^n-1}{2}+\frac{x^{n-2}d}{2}>\frac{x^{n-2}d}{2}$. Then it is clear that we have $n=2,d=p$. This implies $(x+\sqrt{p})^2=2p+1+\sqrt{p(p+1)}$, which is solvable only if $p+1$ is a square number. However, $p+1$ is not square for $p>3$.

    CASE 2. If $d\equiv1\pmod{4}$, then the fundamental unit is of the form $x+y\frac{1+\sqrt{d}}{2}=(x+\frac{y}{2})+\frac{y}{2}\sqrt{d}$, with $x,y$ positive integers. There would exist an integer $n \ge 2$ such that 
    \begin{align*}
        2(2p+1)=&\mathrm{Tr}(\varepsilon)=\mathrm{Tr}\left(\left(x+\frac{y}{2}+\frac{y}{2}\sqrt{d}\right)^n\right)\\
        \geq&2\left(x+\frac{y}{2}\right)^n+2\left(x+\frac{y}{2}\right)^{n-2}\left(\frac{y}{2}\right)^{2}d,
    \end{align*}
    that is,
    \begin{align}
    \label{pinequality2}
        p\geq\frac{\left(x+\frac{y}{2}\right)^n-1}{2}+\frac{\left(x+\frac{y}{2}\right)^{n-2}\left(\frac{y}{2}\right)^{2}d}{2}.
    \end{align}
    Since $p\mid d$ and $x,y$ are positive integers, this inequality does not hold for $y\geq3$.

    If $n=2$, then the inequality (\ref{pinequality2}) is $p\geq\frac{\left(x+\frac{y}{2}\right)^2-1}{2}+\frac{y^{2}d}{8}$, which implies $d\leq 8p$. Note that $d$ is the square-free part of $p(p+1)$ and $d\equiv1\pmod{4}$, we have $p\equiv 3\pmod{4}$. Then $d=3p$ or $7p$ and thus $y=1$. This means that the fundamental unit is of the form $x+y\sqrt{d}=x+\sqrt{3p}$ or $x+\sqrt{7p}$, but it is impossible since $\left(x+y\sqrt{d}\right)^2=2p+1+\sqrt{p(p+1)}$ is unsolvable in both cases. 

    If $n\geq 3$, then the inequality (\ref{pinequality2}) is $p\geq\frac{\left(x+\frac{y}{2}\right)^3-1}{2}+\frac{\left(x+\frac{y}{2}\right)\left(\frac{y}{2}\right)^{2}d}{2}>\frac{3d}{16}$. Using the same method as before, we have $d=3p$, $y=1$ and $x\leq7$. Now, the fundamental unit is of the form $x+y\sqrt{d}=x+\sqrt{3p}$, which is impossible since it is not even a unit for $x\leq7$ and $p\equiv3\pmod{4}$.

    Note that for a prime $\mathfrak{p}$ lying above $p$, we have $\varepsilon=2p+1+2\sqrt{p(p+1)} \in U_{\mathfrak{p}}^{(1)} \setminus U_{\mathfrak{p}}^{(2)}$, where $U_{\mathfrak{p}}^{(i)}$ denotes the $i$\textsuperscript{th} higher unit group in the local field $K_{\mathfrak{p}}$. If $\varepsilon$ were a $p$\textsuperscript{th} power in $K_{\mathfrak{p}}$, then we have $\varepsilon=\alpha^p$ for some $\alpha\in K_\mathfrak{p}$. Since $\varepsilon\in U_\mathfrak{p}^{(1)}$ and the norm $N(\mathfrak{p})=p$, we have $1\equiv\varepsilon\equiv\alpha^p\equiv\alpha\pmod{\mathfrak{p}O_K}$, which implies $\alpha \in U_{\mathfrak{p}}^{(1)}$. However, according to \cite[Proposition 9]{serre}, raising to the $p$\textsuperscript{th} power induces an isomorphism $U_{\mathfrak{p}}^{(1)} \cong U_{\mathfrak{p}}^{(3)}$. This implies $\varepsilon = \alpha^p\in U_{\mathfrak{p}}^{(3)} \subset U_{\mathfrak{p}}^{(2)}$, which contradicts the fact that $\varepsilon \notin U_{\mathfrak{p}}^{(2)}$. Thus, $\varepsilon$ is indeed not a $p$\textsuperscript{th} power in $K_{\mathfrak{p}}$.
    
    It remains to show that $p$ does not divide the class number $h_K$. 
    
    If $p\equiv 3\pmod{4}$, then $p(p+1)$ is divisible by $4$ and $d\leq \frac{p(p+1)}{4}$. By \cite[Theorem (a)]{Le}, we have $h_K\leq\frac{\sqrt{4d}}{2}\leq\sqrt{d}< p$. 

    If $p\equiv1\pmod{4}$, then we have $2\mid d$. Since 2 ramifies in $K$, by \cite[Corollary 2]{Louboutin2004}, it leads to 
    $$L(1,\chi_K)\leq \frac{\log d+\kappa_2}{4},$$
    where $\kappa_2:=2+\gamma-\log\pi=1.432\cdots$, and $\gamma=0.577\cdots$ is the Euler's constant. 

    Then by Dirichlet's class number formula $h_K=\frac{L(1,\chi_K)\sqrt{d}}{2R_K}$ (note that $R_K=\log(\varepsilon)$ is the regulator of $K$), we have
    \begin{align*}
        h_K &<\frac{\log (p(p+1))+2}{4}\cdot\frac{\sqrt{p(p+1)}}{2\log(2p+1+2\sqrt{p(p+1)})}\\
        &<\frac{\log (p(p+1))+2}{4}\cdot\frac{\sqrt{p(p+1)}}{2\log(2\sqrt{p(p+1)})}\\
        &<\frac{\log (p(p+1))+2}{4\log(4p(p+1))}\cdot\sqrt{p(p+1)}\\
        &<\left(\frac{1}{4}+\frac{1}{2\log(4p(p+1))}\right)\cdot\sqrt{p(p+1)}\\
        &<p.
    \end{align*}
    Hence we have $p\nmid h_K$. This completes the proof. 
\end{proof}

\subsection{Pairs of $p$-rational real quadratic fields in arithmetic progression}
In \cite{CLS}, Chattopadhyay, Laxmi and Saikia investigated the existence of consecutive real quadratic fields that are $p$-rational. Specifically, they established the $p$-rationality of pairs such as $\left(\mathbb{Q}\left(\sqrt{n}\right),\mathbb{Q}\left(\sqrt{n+1}\right)\right)$, where $n=p^2-2$ or $n=p^2+1$.

However, as illustrated by Question 1 in \cite{CLS} and related conjectures in the literature (e.g., regarding class number divisibility), it is natural to extend this inquiry to sequences of quadratic fields in arithmetic progression, i.e., fields of the form $\mathbb{Q}(\sqrt{d}),\mathbb{Q}(\sqrt{d+b}),\cdots,\mathbb{Q}(\sqrt{d+nb})$ with a given positive integer $n$. Indeed, in the imaginary case, Chattopadhyay, Laxmi and Saikia \cite{CLS} considered families of the form $\mathbb{Q}\left(\sqrt{-p(p-1)}\right),\mathbb{Q}\left(\sqrt{-p(p-n)}\right)$, which correspond to such progressions.

Motivated by this broader perspective, we observe that Theorem \ref{p(p+1)} allows us to construct such examples in the real case. Benmerieme and Movahhedi \cite{BM2021} proved that the fields $\mathbb{Q}\left(\sqrt{p(p-1)}\right),\mathbb{Q}\left(\sqrt{p(p-2)}\right)$ and $\mathbb{Q}\left(\sqrt{p(p+2)}\right)$ are $p$-rational for $p>3$. Combining their results with our Theorem \ref{p(p+1)}, we identify new pairs of $p$-rational real quadratic fields. 

\begin{corollary}
    For any prime $p>3$, the following pairs of real quadratic fields are simultaneously $p$-rational:
    \begin{enumerate}
        \item $\left(\mathbb{Q}\left(\sqrt{p(p-2)}\right),\mathbb{Q}\left(\sqrt{p(p-1)}\right)\right);$
        \item $\left(\mathbb{Q}\left(\sqrt{p(p+1)}\right),\mathbb{Q}\left(\sqrt{p(p+2)}\right)\right).$
    \end{enumerate}
\end{corollary}

\subsection{New families of $p$-rational triquadratic number field}

Now, we apply Theorem \ref{p(p+1)} to construct a new family of $p$-rational triquadratic number fields. The study of such fields was discussed by Benmerieme and Movahhedi \cite{BM2021}, who suggested considering triquadratic fields of the form $\mathbb{Q}\left(\sqrt{p(p-2)}, \sqrt{p(p+2)},\allowbreak \sqrt{-\alpha}\right)$. Following this direction, Koperecz \cite{Ko} successfully proved that the triquadratic number field $\mathbb{Q}\left(\sqrt{p(p-2)},\allowbreak \sqrt{p(p+2)}, \sqrt{-1}\right)$ is $p$-rational for infinitely many primes $p$, which relied on the existence of large square factors of $p-2$ and $p+2$. Chattopadhyay, Laxmi and Saikia \cite{CLS} generalized this result to fields $\mathbb{Q}\left(\sqrt{p(p-2)}, \sqrt{p(p+2)}, \sqrt{-\alpha}\right)$ using the same analytic method.

To construct our new family of fields, we require the large square factor result established in \cite{CLS}.

\begin{proposition}[\cite{CLS}, Proposition 4]
    \label{factors}
    Given a real number $A>0$ and finitely many nonzero integers $r_1,\cdots,r_s$, there exist infinitely many prime numbers $p$ such that for each $i \in \{1,\dots, s\}$, $p-r_i$ possesses a factor $m_i^2$ with  $m_i > (\log p)^A$.
\end{proposition}

We now consider a variant of the fields suggested in \cite{BM2021}.

\begin{proposition}
\label{trialpha}
    Let $\beta$ be a positive integer such that the quadratic field $\mathbb{Q}\left(\sqrt{-\beta}\right)$ is $p$-rational. There exist infinitely many primes $p$ such that the triquadratic number field $K_\beta=\allowbreak\mathbb{Q}\left(\sqrt{p(p+1)},\allowbreak\sqrt{p(p-1)},\sqrt{-\beta}\right)$ is $p$-rational. 
\end{proposition}

\begin{proof}
        According to Greenberg \cite[Proposition 3.6]{Greenberg}, for a given positive integer $\beta$, in order to prove the $p$-rationality of $K_\beta=\mathbb{Q}\left(\sqrt{p(p+1)},\sqrt{p(p-1)},\sqrt{-\beta}\right)$ for infinitely many primes $p$, we need to show that all quadratic subfields of $K_\beta$ are $p$-rational. For $p>3$, the $p$-rationality of $\mathbb{Q}\left(\sqrt{p(p-1)}\right)$ and $\mathbb{Q}\left(\sqrt{(p+1)(p-1)}\right)$ has been proved in \cite{Greenberg,BR,BM2021}, and the $p$-rationality of $\mathbb{Q}\left(\sqrt{p(p+1)}\right)$ is proved in Theorem \ref{p(p+1)}. Hence, we only need to show that there exist infinitely many primes $p$ such that $K_0=\mathbb{Q}\left(\sqrt{-\beta}\right),K_1=\mathbb{Q}\left(\sqrt{-\beta p(p+1)}\right),K_2=\mathbb{Q}\left(\sqrt{-\beta p(p-1)}\right)$ and $K_3=\mathbb{Q}\left(\sqrt{-\beta(p+1)(p-1)}\right)$ are all $p$-rational. 

    Given $r_1=-1,r_2=1$ and $A=2$, let $\mathcal{P}$ denote the infinite set of primes $p$ as given by Proposition \ref{factors} with $p>3$. By this proposition, there exist integers $m_1, m_2 > (\log p)^2$ such that $m_1^2 \mid (p+1)$ and $m_2^2 \mid (p-1)$. For $p\in\mathcal{P}$, we have
    \begin{align*}
       d_{K_1} &\le \frac{4\beta p(p+1)}{m_1^2} < \frac{4\beta p(p+1)}{(\log p)^4}, \\
        d_{K_2} &\le \frac{4\beta p(p-1)}{m_2^2} < \frac{4\beta p(p-1)}{(\log p)^4}, \\
        d_{K_3} &\le \frac{4\beta (p+1)(p-1)}{m_1^2 m_2^2} < \frac{4\beta (p+1)(p-1)}{(\log p)^8}.
    \end{align*}
    Using Louboutin's bound for imaginary quadratic field as stated in Proposition \ref{hk for imaginary field}, for $p\in\mathcal{P}$, we have
    \begin{align*}
        h_{K_1}&\leq\frac{6}{4\pi}\sqrt{\frac{4\beta p(p+1)}{(\log p)^4}}\left(\log \left(\frac{4\beta p(p+1)}{(\log p)^4}\right)+\frac{3}{2}\right)\ll \frac{p}{\log p},\\
        h_{K_2}&\leq\frac{6}{4\pi}\sqrt{\frac{4\beta p(p-1)}{(\log p)^4}}\left(\log \left(\frac{4\beta p(p-1)}{(\log p)^4}\right)+\frac{3}{2}\right)\ll \frac{p}{\log p},\\
        h_{K_3}&\leq\frac{6}{4\pi}\sqrt{\frac{4\beta (p+1)(p-1)}{(\log p)^8}}\left(\log \left(\frac{4\beta(p+1)(p-1)}{(\log p)^8}\right)+\frac{3}{2}\right)\ll \frac{p}{(\log p)^3},
    \end{align*}
    with absolute and effective implied constants. 

    Then for sufficiently large prime $p\in\mathcal{P}$, $h_{K_1},h_{K_2}$ and $h_{K_3}$ are not divisible by $p$, and thus the quadratic subfields $K_1=\mathbb{Q}\left(\sqrt{-\beta p(p+1)}\right),K_2=\mathbb{Q}\left(\sqrt{-\beta p(p-1)}\right)$ and $K_3=\mathbb{Q}\left(\sqrt{-\beta (p+1)(p-1)}\right)$ are all $p$-rational. 

    Therefore, there exist infinitely many primes $p$ such that the triquadratic field $K_\beta=\mathbb{Q}\left(\sqrt{p(p+1)},\sqrt{p(p-1)},\sqrt{-\beta}\right)$ is $p$-rational.
\end{proof}
\begin{remark}
    It is worth noting that the case $\beta=1$ corresponds to a family similar to those studied in \cite{CLS} (specifically, the case $\alpha=2$). However, a crucial distinction lies in the properties of the real biquadratic subfield $F = \mathbb{Q}\left(\sqrt{p(p+1)}, \sqrt{p(p-1)}\right)$.
In \cite{CLS}, the $p$-rationality of $F$ is derived as a consequence of the large square factor conditions.
In contrast, the real biquadratic subfield $F$ is unconditionally $p$-rational for all primes $p > 3$. While this property of $F$ was briefly noted in \cite{Benmerieme2021}, it is explicitly established by combining our Theorem \ref{p(p+1)} for the subfield $\mathbb{Q}\left(\sqrt{p(p+1)}\right)$ with the results of \cite{BM2021} for $\mathbb{Q}\left(\sqrt{p(p-1)}\right)$.
\end{remark}
    Let us consider the specific case $\beta=1$, which yields the triquadratic field $K = \mathbb{Q}\left(\sqrt{p(p+1)}, \sqrt{p(p-1)}, \sqrt{-1}\right)$. This corresponds to the family defined by $\alpha=2$ in the general framework of \cite[Corollary 2]{CLS}.
    
    We emphasize that the set of prime numbers $p$ for which $K$ is $p$-rational is generally different from the set obtained for $L=\mathbb{Q}\left(\sqrt{p(p+2)},\allowbreak \sqrt{p(p-2)},\allowbreak \sqrt{-1}\right)$.
    For instance, there may exist primes $p$ where $K$ is $p$-rational but $L$ is not. 
    Therefore, by establishing the $p$-rationality of this new family, we successfully provide new explicit examples of $p$-rational fields that were not covered by previous specific constructions.

\section*{Acknowledgment}
We would like to express our great gratitude to Abbas Chazad Movahhedi for introducing us to Greenberg's conjecture and for his consistent support. 

\section*{Funding}
The authors are supported by National Natural Science Foundation of China (No. 12231009).

\end{document}